\documentclass[12pt,makeidx]{amsart}

\usepackage{amssymb}
\usepackage{txfonts}

\usepackage{nag}

\usepackage{enumerate,color}
\usepackage[pagebackref,colorlinks,linkcolor=red,citecolor=blue,urlcolor=blue,hypertexnames=true]{hyperref}
\usepackage{url}

\renewcommand{\subset}{\subseteq}
\newtheorem{theorem}            {Theorem}[section]
\newtheorem{corollary}          [theorem]{Corollary}
\newtheorem{proposition}        [theorem]{Proposition}

\newtheorem{lemma}              [theorem]{Lemma}
\newtheorem{remark}         [theorem]{Remark}

\newcommand{\df}[1]{{\bf{#1}}{\index{#1}}}

\def\cFc{\mathfrak F_c}
\def\bZ{\mathbb Z}
\def\tT{{\tt T}}
\def\tH{{\tt H}}
\def\tB{{\tt B}}
\def\ttJ{{\tt J}}

\def\obD{\overline{\mathbb D}}

\def\bZ{\mathbb Z}
\def\cK{\mathcal K}
\def\cP{\mathcal P}

\def\cJ{\mathcal J}
\def\cH{\mathcal H}
\def\cE{\mathcal E}
\def\cF{\mathcal F}

\def\ol{\overline{\lambda}}

\def\tN{\tilde{N}}
\def\tJ{\tilde{J}}
\def\tW{\tilde{W}}
\def\tE{\tilde{E}}

\def\cY{\mathcal Y}
\def\cV{\mathcal V}
\def\fU{\mathfrak{U}}
\usepackage{color}
\usepackage{tikz}

\makeindex

\begin{document}

%% topmatter
\title[Lifting $3$-isometries]{A lifting theorem for $3$-isometries}

\author[McCullough]{Scott McCullough${}^*$}
\address{Scott McCullough, Department of Mathematics\\
  University of Florida, Gainesville %\\
   % Box 118105\\
   %  Gainesville, FL 32611-8105\\
   %  USA
   }
   \email{sam@math.ufl.edu}
\thanks{${}^*$ Partially supported by NSF grant  DMS-1101137}

\author[Russo]{Benjamin Russo}
 \address{Benjamin Russo, Department of Mathematics\\
  University of Florida, Gainesville %\\
   % Box 118105\\
   %  Gainesville, FL 32611-8105\\
   %  USA
   }
   \email{russo5@math.ufl.edu}

%BELOW ARE CLASSES FOR THE ANNALS PAPER **

%\date{\today}

\maketitle

\begin{abstract}
  An operator $T$ on Hilbert space is a $3$-isometry if 
  there exists operators $B_1$ and $B_2$ such that
  $T^{*n}T^n=I+nB_1+n^2B_2.$   An operator $J$ is
  a Jordan operator if it the sum of a unitary $U$
 and nilpotent $N$ of order two which commute. If $T$ is a $3$-isometry and $c>0,$
  then  $I-c^{-2} B_2 + sB_1 + s^2B_2$ is positive semidefinite for all 
 real $s$ if and only if  it is the restriction
  of a Jordan operator $J=U+N$ with the norm of $N$ at most $c.$ 
  As a 
  corollary, an analogous result for $3$-symmetric operators,
  due to Helton and Agler, is recovered. 
\end{abstract}

\section{Introduction}
  Let $B(H)$ denote the bounded operators
  on the (complex) Hilbert space $H$.  An operator $T\in B(H)$ is 
 a \df{$3$-isometry} if \index{$T$ a $3$-isometry}
\[
  T^{*3} T^3 -3 T^{*2} T^2 + 3 T^* T - I =0.
\]
  Equivalently $T$ is a $3$-isometry if and only if
  there exist operators $B_1(T^*,T)$ and $B_2(T^*,T)$ such that \index{$B_j(T^*,T)$}
  for all natural numbers $n$
\begin{equation}
 \label{eq:3iso}
  T^{*n}T^n = I + n B_1(T^*,T) + n^2 B_2(T^*,T).
\end{equation}
  In this case it is straightforward to verify that
\begin{equation}
 \label{eq:B2}
    2 B_2(T^*,T) = T^{*2}T^2-2T^{*}T + I
\end{equation}
 and
\begin{equation}
 \label{eq:B1}
 2 B_1(T^*,T) = -T^{*2}T^2 + 4T^{*}T-3.
\end{equation}
Evidently, each $B_j(T^*,T)$ is selfadjoint.
  
 From  Equation \eqref{eq:3iso},
  it is evident that that $\|T^n\|^2$ is bounded by 
  a quadratic in $n$.  It follows from the spectral radius 
  formula that the spectrum of $T,$ denoted $\sigma(T)$,  is a subset of $\obD,$
 the closed unit disc.  If $T$ is invertible, then Equation \eqref{eq:3iso} holds
  for all integers $n$. In particular, $T^{-1}$ is also
  a $3$-isometry and hence $\sigma(T^{-1}) \subset \obD$.
  Thus, in this case, $\sigma(T)$ is a subset of the
  unit circle.  As will be seen later, if 
  $T$ is not invertible, then in fact $\sigma(T)=\obD$
  \cite{aglerstankusi}
  (or see Lemma \ref{lem:specTnotinv}).

  Likewise, an operator $\tT$ on a Hilbert space $\tH$ 
  is a \df{$3$-symmetric operator} \index{$\tT$ a $3$-symmetric operator}\index{$\tH$ Hilbert space $\tT$ acts on}
  if there exists operators $\tB_j(\tT^*,\tT)$ on $\tH$ such that 
  \begin{equation}
 \label{eq:3sym}
 \exp(-is\tT^*) \exp(is\tT) = I + s \tB_1(\tT^*,\tT) + s^2 \tB_2(\tT^*,\tT)
\end{equation}
 for all real $s$.  Evidently, if $\tT$ is $3$-symmetric, then
  $T=\exp(i\tT)$ is a $3$-isometry.   Helton introduced
  $3$-symmetric operators as both a generalization of
  selfadjoint operators and as a class of non-Normal operators
  for which a viable spectral theory exists. In a 
  series of papers (\cite{heltonbulletin} \cite{heltonpaper} \cite{heltonpapererata})  Helton 
  modeled these operators as multiplication $t$ on 
  a Sobolev space and showed, under some additional
  hypotheses,  that they are the restriction to
  an invariant subspace of a Jordan operator (of order two) as explained
  below. In \cite{aglerthreesym} the 
  connection between Jordan operators and $3$-symmetric
 operators was established in general.   See also the
  references 
 \cite{ballhelton} \cite{aglerstankusi} \cite{aglerstankusii} \cite{aglerstankusiii}
  \cite{richter} \cite{mccullough}.
 
%  Suppose $U$ and $N$ are operators on 
%  a Hilbert space $K$ and $U$ is 
%  unitary and $N$ is nilpotent of order two so that
%  $N^2=0.$ 
%  If $U$ and $N$ commute, then $J=U+N$ is a $3$-isometry.
 % Further, if $H\subset K$ is invariant for $J$, then 
%  $T= J|_H$ is a $3$-isometry. 
%%  A converse of this fact is a basic
 %% version of the main result of this 
%  %article. 

  Given a positive number $c$, let $\cFc$ denote those \index{$c$ a parameter for the class $\cFc$}
  $3$-isometries $T$ such that \index{$\cFc$ those $3$-isometries satisfying \eqref{eq:Q}}
 % in Equation \eqref{eq:3iso},
   the quadratic 
\begin{equation}
 \label{eq:Q}
  Q(T,s) = [I-\frac{1}{c^2} B_2(T^*,T)] + sB_1(T^*,T) +s^2 B_2(T^*,T) 
\end{equation}
 is positive semidefinite for all real numbers $s$. 
% (When $T$ is understood, $Q(T,s)$ will sometimes
%  be abbreviated to $Q(s)$.)
 It turns out there are $3$-isometries which do not
  belong to any of these classes. \index{$U$ a unitary}
 If $U$ is a unitary operator on the Hilbert space $\cF$, then the \df{Jordan operator}\index{$J$ a Jordan operator as in \eqref{eq:Jordanc}}
\begin{equation}
 \label{eq:Jordanc}
  J= \begin{pmatrix} U  & c U \\ 0  & U \end{pmatrix}
\end{equation}
 acting on $K=\cF\oplus \cF$ is a canonical member of the class $\cFc$.  
% It is staightforward to show that if $T$ is 
%  $3$-isometry then the spectrum of $T$ is the
%  closed unit disc (in the case $T$ is not invertible),
% or a subset of the unit circle (in the case $T$ is invertible).
 Indeed, it is readily verified that \index{$Q(T,s)$}
\[
 Q(J,s) = \begin{pmatrix} I & scI \\ scI & s^2 c^2 I\end{pmatrix} \succeq 0.
\]
  Here, for an operator $A$ on Hilbert space, $A\succeq 0$ means $A$ is
 positive semidefinite. 
 Moreover, if $T$ is an operator  acting on the Hilbert space  $H$ and 
  there  is an isometry $V:H\to K$ such that 
  $VT=JV$, then $ T^{*n}T^n = V^* J^{*n}J^n V$ for all $n$.
  It follows that $T$ is a $3$-isometry and further,
\[
 Q(T,s) = V^* Q(J,s)V \succeq 0
\]
 so that $T\in\cFc$.  The following is the main result of this article.

\begin{theorem}[$3$-isometric lifting theorem]
 \label{thm:main}
   An operator $T$ on a Hilbert space $H$ is in the class
  $\cFc$ if and only if there is an operator $J$ of
  the form in Equation \eqref{eq:Jordanc} acting on
  a Hilbert space $K$ and an isometry $V:H\to K$ such that
 $VT=JV$.

  If $T$ is invertible, then, necessarily, $VT^{-1} = J^{-1}V.$
  Moreover, in this case,  the spectrum of $T$  is a subset of the
  unit circle, and $J$ can be chosen so that 
  $\sigma(T)=\sigma(J)$. 
\end{theorem}

  The $3$-symmetric lifting Theorem of Helton and Agler is
 fairly easily seen to be a consequence of Theorem \ref{thm:main}.
 The details are in Section \ref{sec:3symmetrics}. 
  The proof of Theorem \ref{thm:main} for the case
  that $T$ is invertible uses Arveson's complete
 positivity machinery in the form of a version of the Arveson 
  Extension Theorem  and operator-valued Fejer-Riesz Factorization.
  The proof of the Arveson Extension Theorem along with the
  needed background on
  the  theory of completely positive maps appears in 
  Section \ref{sec:arv} gives the 
  The proof of the lifting theorem for invertible
  $3$-isometries appears in Section \ref{sec:inv}. 
  The reduction of the general case of Theorem \ref{thm:main}
  to the invertible case is the topic of Section \ref{sec:notinv}.
  A functional calculus argument establishes the spectral condition, $\sigma(T)=\sigma(J)$,
  in Section \ref{sec:spectral}.

\section{Completely Positive Maps and the Arveson Extension Theorem}
 \label{sec:arv}
  In this section some of  Agler's hereditary calculus machinery based upon
  the Arveson Extension Theorem is reviewed. Let $n$ and $N$
  be given positive integers.  An \df{hereditary polynomial} $p(x,y)$ 
  of size $n$ and degree at most $N$ 
  in noncommuting (invertible) variables $x$ and $y$ is a polynomial 
  of the form \index{$p$ an hereditary polynomial}
\begin{equation}
 \label{eq:hpoly}  
  p(x,y) = \sum_{\alpha,\beta=-N}^N p_{\alpha,\beta} y^\alpha x^\beta.
\end{equation}
  Here the sum is finite and the $p_{\alpha,\beta}$ are $n\times n$
  matrices over $\mathbb C$.  Such a polynomial is evaluated at an
  invertible  operator $T$ by
\[
  p(T^*,T) = \sum p_{\alpha,\beta}\otimes  T^{*\alpha} T^\beta.
\]
  Let $\cP_n$ denote the collection \index{$\cP_n$ - $n\times n$ matrix hereditary polys}
 of hereditary polynomials of size $n$ and 
  let $\cP = (\cP_n)_n$ denote the collection of all  hereditary polynomials.
 \index{$\cP$ - Hereditary polys of all sizes}

Given an operator $T$, let $\cH(T)$ \index{$\cH(T)$ the span of $\{T^{*\alpha}T^\beta\}$} denote the
  span by $\{T^{*\alpha}T^\beta: \alpha,\beta\in\mathbb Z\}.$
  Given an operator $J$ on the Hilbert space $K$, if $p(J^*,J)\succeq 0$ implies $p(T^*,T)\succeq 0$,
  then the mapping $\rho:\cH(J)\to \cH(T)$ given by
  $\rho(p(J,J^*))= p(T^*,T)$ is well defined. Let $M_n$ denote the $n\times n$ matrices.
  If in addition, for each $n$ the mapping
 $1_n\otimes \rho: M_n\otimes \cH(J)\to M_n\otimes \cH(T)$ obtained by applying $\rho$
  entry-wise is positive, then $\tau$ is \df{completely positive}.

\begin{theorem}[Arveson Extension Theorem \cite{aglerhereditary}]
 \label{thm:arvext}
  Suppose $T$ and $J$ are invertible operators on Hilbert spaces $H$ and $K$ respectively.
  There is a Hilbert space $\cK,$
  a representation $\pi:B(K)\to B(\cK)$ and an isometry $V:H\to \cK$ such that $VT^j = \pi(J)^j V$ for all $j\in\mathbb Z$  if 
  and only if the mapping $\rho:\cH(J)\to \cH(T)$ is completely positive. 
\end{theorem}

\begin{proof}
  %The hypothesis that  $p\in\cP$ and $p(J^*,J)\succeq 0$ 
 % implies $p(T,T^*)\succeq 0$
 % is equivalent to complete positivity of the mapping  
  Since $\rho:\cH(J)\to\cH(T)$ determined by 
  $\rho(J^{*\alpha} J^\beta) = T^{*\alpha} T^\beta$ 
  is (well defined) and completely positive, 
  by Theorem ?? in \cite{paulsenbook}, there is
  a Hilbert space $\cK,$ a representation $\pi:B(K)\to B(\cK)$ 
  and an isometry $V:H\to\cK$ such that 
\[
  T^{*\alpha}T^\beta =\rho(J^{*\alpha} J^\beta) = V^* \pi(J^{*\alpha} J^\beta)V.
\]
  
  For each $\beta\in\mathbb Z$,
\[
 \begin{split}
   V^* \pi(J)^{*\beta} \pi(J)^\beta V = & T^{*\beta}T^\beta \\ 
     = & V^*\pi(J)^{*\beta} VV^*  \pi(J)^\beta V.
 \end{split}
\]
  Thus, as $I-VV^*$ is a projection, 
\[
  V^*\pi(J)^{*\beta} (I-VV^*)^2 \pi(J)^\beta V =0
\]
 and therefore $(I-VV^*)\pi(J)^\beta V=0$. 
 Consequently,  $VV^* \pi(J)^\beta V = \pi(J)^\beta V.$   It follows that, for each $\beta$,
\[
  VT^\beta = VV^* \pi(J)^\beta V = \pi(J)^\beta V.
\]

  The proof of the converse is routine.  
\end{proof}

\subsection{Symmetrization}
  In this section Agler's symmetrization technique \cite{aglerthreesym},
  which leads to 
  a strong variant of Theorem \ref{thm:arvext},
  is reviewed.    An operator $J$ is \df{symmetric} if $\exp(it)J$ is 
  unitarily equivalent to $J$  for each real number $t$.
  Given an invertible operator $T$, let $\cH_s(T)$  denote
  the span of  $\{T^{*\alpha}T^\alpha : \alpha\in\mathbb Z\}$.
  \index{$S$ - bilateral shift}
  \index{$\cH_s(T)$ the span of $\{I,T^*T, T^{*2}T^2\}$}
 
\begin{proposition}[Agler]
 \label{prop:sym}
  Suppose $T$ and $J$ are invertible operators on Hilbert spaces $H$ and $K$ 
  respectively.
  If $J$ is symmetric and the mapping $\rho:\cH_s(J)\to \cH_s(T)$ 
  determined by
  $\rho(J^{*\alpha} J^\alpha) = T^{*\alpha}T^\alpha$ is (well defined and) 
  completely positive, then there is a Hilbert space $\cK$, a representation
  $\pi:B(K)\to B(\cK)$ and an isometry $V$ such that $VT^j = \pi(J)^j V$
  for all $j\in\mathbb Z$.
\end{proposition}

  The proof of this Proposition occupies the remainder of this section.
  Let $S$ denote the \df{bilateral shift} operator on $L^2=L^2(\mathbb T)$.  \index{$S$ the bilateral shift}
  Since $S$ is symmetric, it is readily seen that, for any \index{$L^2=L^2(\mathbb T)$}
  operator $T,$ the operator $\tilde{T}=T\otimes S$ \index{$\tilde{T}=T\otimes S$}
     acting on $H\otimes L^2$ is also symmetric.
  Moreover, if $T\in\cFc$, then so is $\tilde{T}$. 
  Given $p\in\cP$ as in Equation \eqref{eq:hpoly} let 
  $p^s$ denote its \df{symmetrization}, \index{$p^s$ the symmetrization of the poly $p$}
\[
  p^s = \sum p_{\alpha,\alpha} y^\alpha x^\alpha.
\]

\begin{lemma}
  \label{lem:sym0} 
    If $J$ is symmetric, $q\in\cP$ and
    $q(J)\succeq 0$, then $q^s(J)\succeq 0$.

   Let  $T$ be a given operator on the Hilbert space $H$ and \index{$W:H\to H\otimes L^2$}
   let $W:H\to H\otimes L^2$ denote the isometry $Wh= h\otimes 1$.  If
  $P\in\cP_n$, then 
\[
  P^s(T^*,T) = (I_N\otimes W)^* P(\tilde{T}^*,\tilde{T})(I_N\otimes W).
\]
%
% then 
%    for each pair $h,f\in H$ and $p\in\cP_1$,
% \[
%    \langle p(\tilde{T},\tilde{T}^*) h\otimes 1, f\otimes 1\rangle
%      = \langle p^s(T,T^*)h,f\rangle.
% \]
%  If $P \in \cP_n$ and $h_1,\dots,h_n, f_1,\dots,f_n \in H$, then
%\[
%   \langle  P(\tilde{T},\tilde{T}^*) \begin{pmatrix} h_1\otimes 1 \\ \vdots \\ h_n\otimes 1 \end{pmatrix},
%   \begin{pmatrix} f_1\otimes 1 \\ \vdots \\ f_n\otimes 1\end{pmatrix} \rangle
%   =\langle  P^s({T},{T}^*) \begin{pmatrix} h_1 \\ \vdots \\ h_n \end{pmatrix},
%   \begin{pmatrix} f_1  \\ \vdots \\ f_n \end{pmatrix} \rangle
%\]
\end{lemma}

\begin{proof}
  For each $t$ there is a unitary operator $U_t$ such that 
  $e^{it}J=U_t^* J U_t$. Letting $n$ denote the size of $q$
  (so that $q\in\cP_n$),  it follows that
\[
  q(e^{-it}J^*, e^{it}J) = (I_n\otimes U_t^*) q(J^*,J) (I_n\otimes U_t) \succeq 0.
\]
 Hence,
\[
  q^s(J^*,J) =\frac{1}{2\pi} \int_0^{2\pi}   q(e^{-it}J^*,e^{it}J) \, dt \succeq 0.
\]

  To prove the second part, write $p\in\cP_1$ as in Equation
  \eqref{eq:hpoly} in which case,
\[
 \begin{split}
  \langle p(\tilde{T},\tilde{T}^*)Wh,Wf\rangle = &
  \langle p(\tilde{T},\tilde{T}^*) h\otimes 1,f \otimes 1\rangle \\
   = &\langle \sum_{\alpha,\beta}  p_{\alpha,\beta} T^\alpha h \otimes z^\alpha,
         T^{\beta}f \otimes z^\beta \rangle \\
   = & \langle \sum_{\alpha} p_{\alpha,\alpha} T^{*\alpha} T^{\alpha}h,f\rangle \\
   =&\langle p^s(T,T^*)h,f\rangle.
 \end{split}
\]
  Applying the result for $p\in\cP_1$ entry-wise to $P$ completes the proof.
\end{proof}

\begin{lemma}
 \label{lem:sym1}
   Suppose $T$ and $J$ are invertible operators on Hilbert spaces $H$ and $K$ 
  respectively.
  If $J$ is symmetric and the mapping $\rho:\cH_s(J)\to \cH_s(T)$ 
  determined by
  $\rho(J^{*\alpha} J^\alpha) = T^{*\alpha}T^\alpha$ is (well defined and) 
  completely positive, then the mapping $\tilde{\rho}:\cH(J)\to\cH(\tilde{T})$
  determined by 
\[
  \tilde{\rho}(J^{*\alpha} J^\beta) = \tilde{T}^{*\alpha} \tilde{T}^\beta
\]
 is also (well defined and) completely positive.
\end{lemma}

\begin{proof}
  Fix a positive integer $n$ and a $p\in\cP_n$.   
  In particular, $p(T^*,T)$ acts on $\mathbb C^n\otimes H$. 
 % Let $\{e_1,\dots,e_n\}$ denote the
 % standard orthonormal basis for $\mathbb C^n$. 
  Given a positive integer $N$
  consider the $(2N+1)\times (2N+1)$  matrix 
  whose entries are $n\times n$ matrix polynomials
\[
  P = \begin{pmatrix}  (I_n\otimes y^j) p(x,y) (I_n\otimes x^k) \end{pmatrix}_{j,k=-N}^N.
\]
  Here $I_n$ is the $n\times n$ identity matrix. Thus, the $(j,k)$ entry
  of $P(T^*,T)$ is the operator on $\mathbb C^n\otimes H$ given by
\[
  (I_n\otimes T^{*j}) p(T^*,T) (I_n\otimes T^k).
\]

  Viewing $P(T,T^*)$ as an operator on
  $(\mathbb C^n\otimes H)\otimes \mathbb C^{2N+1}$,
  let $\{e_{-N},\dots,e_0,\dots,e_N\}$ denote
  the corresponding standard basis for $\mathbb C^{2N+1}$.
  Given a vector $h =\sum h_a \otimes e_a \in  (\mathbb C^n\otimes H)\otimes \mathbb C^{2N+1}$,
  %  Given vectors $h_{a,b}\in H$, let $h_a = \sum h_{a,b} e_b$ 
% ($-N\le a\le N$ and $1\le b\le n$)
%  and $h$ the vector 
%\[
%  h = \begin{pmatrix} h_{-N} \\ \vdots \\ h_0  \\ \vdots
%    \\ h_N \end{pmatrix}.
%\]
 % With these notations, 
  an application of Lemma \ref{lem:sym0} gives,
\begin{equation}
 \label{eq:symkey}
 \begin{split}
    \langle P^s({T}^*,T) h,h\rangle 
     = &  \langle P(\tilde{T}^*,\tilde{T}) h\otimes 1, h\otimes 1\rangle \\
    =  &\sum \langle
       (I\otimes \tilde{T}^{*j}) p(\tilde{T}^*,\tilde{T}) (I\otimes \tilde{T}^k) h_k\otimes 1,
         h_j\otimes 1\rangle \\
    = & \sum \langle p(\tilde{T}^*,\tilde{T}) (I\otimes \tilde{T}^k) h_k\otimes 1, (I\otimes \tilde{T}^j) h_j\otimes 1\rangle \\
    = & \sum \langle p(\tilde{T}^*,\tilde{T}) [(I\otimes {T}^k) h_k]\otimes z^k, [(I\otimes {T}^j) h_j] \otimes z^j \rangle \\
    = & \langle p(\tilde{T}^*,\tilde{T}) \sum_{k=-N}^N [(I\otimes T^k)h_k] \otimes z^k, \sum_{j=-N}^N [(I\otimes T^j) h_j] \otimes z^j \rangle.
 \end{split}
\end{equation}

   Now suppose that $p(J^*,J)\succeq 0$. It then follows that $P(J^*,J)\succeq 0$ and thus $P^s(J^*,J)\succeq 0$.
   The hypotheses imply $P^s(T^*,T)\succeq 0.$ 
   From Equation \eqref{eq:symkey} and the fact that sums of the form $\sum_{j=-N}^N T^j h_j\otimes z^j$ are
   dense in $H\otimes L^2$ (since $T$ is invertible), it follows that $p(\tilde{T}^*,\tilde{T})\succeq 0$. 
\end{proof}

\begin{lemma}
 \label{lem:sym2}
   Suppose $T\in B(H)$ is invertible. 
   If $p\in\cP$ and $p(\tilde{T}^*,\tilde{T})\succeq 0$, then
   $p(T^*,T)\succeq 0$. In particular, the canonical mapping
   $p(\tilde{T}^*,\tilde{T})\mapsto p(T^*,T)$ is well defined.
\end{lemma}

\begin{proof}
 Let
\[
  D_N = \frac{1}{2N+1} \sum_{j=-N}^N e^{ijt} \in L^2(\mathbb T).
\]
 If $h,f\in H$, then for $\alpha,\beta\in\mathbb Z$,
\[
 \begin{split}
   \langle \tilde{T}^\alpha h\otimes D_N, \tilde{T}^\beta f\otimes D_N \rangle
   = & \langle T^\alpha h,T^\beta f\rangle \, \langle z^{\alpha-\beta} D_N,D_N\rangle \\
   = & \langle T^\alpha h,T^\beta f\rangle \, \frac{2N+1 -|\alpha-\beta|}{2N+1}.
 \end{split}
\] 
  Thus, if $p\in\cP_1$, then 
\[
 \lim_{N\to\infty} \langle p(\tilde{T}^*,\tilde{T}) h\otimes D_N, f\otimes D_N\rangle 
    = \langle p(T^*,T)h,f\rangle.
\]
  In particular, if $p(\tilde{T}^*,\tilde{T})\succeq 0$, then also $p(T^*,T)\succeq 0$.
 The square matrix version of this implication is readily established and 
  proves the lemma.
\end{proof}

\begin{proof}[Proof of Proposition \ref{prop:sym}]
 From Lemma \ref{lem:sym1}, the mapping $\tilde{\rho}:\cH(J)\to\cH(\tilde{T})$
 (as defined in Lemma \ref{lem:sym1})  is completely positive. 
 On the other hand, from Lemma \ref{lem:sym2}, the canonical  mapping 
  $\tau:\cH(\tilde{T}) \to \cH(T)$
  is also (well defined and) completely positive. Thus, the composition
 $\rho=\tau\circ\ \tilde{\rho}$ is also completely positive.
  The conclusion now follows from the Arveson Extension Theorem,
  Theorem \ref{thm:arvext}.
\end{proof}

\section{Lifting  Invertible $3$-Isometries}
\label{sec:inv}
 In this section Theorem \ref{thm:main} is established in the case
  that $T$ is invertible.  The first step uses Proposition \ref{prop:sym} to prove 
  that if $T$ is invertible, then $T$ lifts to a $J$ of the form
  in Equation \eqref{eq:Jordanc}.  A separate argument, 
  found in Section \ref{sec:spectral},  shows
  that the spectrum of $J$ can be chosen to be the same as that of $T$.

%\subsection{The Lift}
 Given $T\in\cFc$, let $B_0(T^*,T) = I-\frac{1}{c^2}B_2(T^*,T)$. \index{$B_0(T^*,T)=I-\frac{1}{c^2} B_2(T^*,T)$}
   The operator-valued quadratic 
\[
 Q(T,s) = \sum_{j=0}^2 B_j(T^*,T) s^j 
\]
 takes positive semi-definite values. Hence,  (\cite{rosenblum})
  there exists an auxiliary Hilbert space $\cY$ and operators
  $V_0,V_1:H\to \cY$ such that 
 as
\begin{equation}
 \label{eq:Vs}
 Q(T,s) = (V_0 +s V_1)^* (V_0 +s V_1).
\end{equation}

 The following lemma validates the hypotheses of Proposition \ref{prop:sym}.
 As before, let $S$ denote the bilateral shift equal the operator of multiplication by $z=e^{it}$ on $L^2(\mathbb T)$. 
 In particular, $S$ is unitary and 
\begin{equation}
 \label{eq:scriptJ}
  \cJ = \begin{pmatrix} S & cS \\ 0 & S \end{pmatrix}
\end{equation}
 has the form of Equation \eqref{eq:Jordanc}. 
 \index{$\cJ$}
  Recalling the definitions of $B_j(\cJ^*,\cJ)$, straightforward computation shows,
\[
 \begin{split}
   B_0(\cJ^*,\cJ) = & \begin{pmatrix} I & 0 \\ 0 & 0 \end{pmatrix}\\
   B_1(\cJ^*,\cJ)= & \begin{pmatrix} 0 & cI \\ cI & 0\end{pmatrix}\\
   B_2(\cJ^*,\cJ) = & \begin{pmatrix} 0 & 0 \\ 0 & c^2 \end{pmatrix}.
 \end{split}
\]

\begin{lemma}
 \label{lem:factor}
  If $T$ is in the class $\cFc$, then the mapping 
  $\rho:\cH_s(\cJ)\to \cH_s(T)$ determined by
  $\rho(\cJ^{*\alpha} \cJ^\alpha) = T^{*\alpha}T^\alpha$ 
  is (well defined and) completely positive.
\end{lemma}

\begin{proof}
% Recall  $B_0(\cJ^*,\cJ) = I-\frac{1}{c^2} B_2(\cJ^*,\cJ)$ and similarly
%   for $B_0(T^*,T)$. 
 The spaces 
 $\cH_s(\cJ)$ and $\cH_s(T)$ are spanned by the triples 
 $\{B_0(\cJ^*,\cJ), B_1(\cJ^*,\cJ), B_2(\cJ^*,\cJ)\}$ and 
  $\{B_0(T^*,T), B_1(T^*,T), B_2(T^*,T)\}$ respectively,
 since  both $\cJ$ and $T$ are $3$-isometries.  
In particular, for $n$ a positive integer and 
  with $M_n$ denoting the $n\times n$ matrices, an element $X\in M_n \otimes \cH_s(\cJ)$
 has the form,
\[
  X= X_0 \otimes B_0(\cJ^*,\cJ)  + X_1 \otimes B_1(\cJ^*,\cJ) + X_2\otimes B_2(\cJ^*,\cJ)
   \cong \begin{pmatrix}  X_0 & c X_1 \\ c X_1 & c^2 X_2 \end{pmatrix} \otimes I,
\]
  where the $X_j$ are $n\times n$ matrices and $I$ is the
  identity on the space that $\cJ$ acts upon. 
 In particular, if $X \succeq 0$, then each $X_j$ is self adjoint.
  Further, $X\succeq 0$ if and only if 
\[
 Y = \begin{pmatrix} X_0 & X_1 \\ X_1 & X_2 \end{pmatrix}
\]
 is too in which case there exists $n\times 2n$ matrices
 $Y_0$ and $Y_1$ such that $Y_j^* Y_k = X_{j+k}$. 

 To see that $\rho$ is completely positive,
  recall Equation \eqref{eq:Vs} and  observe
\[
 \begin{split}
  1_n\otimes \rho(X) = & \sum X_j \otimes B_j(T^*,T) \\
    = & X_0 \otimes V_0^* V_0 + X_1\otimes (V_0^* V_1 + V_1^* V_0) + X_2\otimes V_1^* V_1\\
    = & (Y_0\otimes V_0 + Y_1\otimes V_1)^* (Y_0\otimes V_0 + Y_1\otimes V_1).
 \end{split}
\]
% On the other hand, if $X\succeq 0$, then 
%\[
% \begin{split}
%  0\preceq & \begin{pmatrix} I \otimes I  & I\otimes I \end{pmatrix}
%  \left [ \begin{pmatrix} X_0 & X_1 \\ X_1 & X_2 \end{pmatrix}
%    \otimes \begin{pmatrix} V_0^* V_0 & V_0^* V_1 \\ V_1^* V_0 & V_1^* V_1\end{pmatrix}
%      \right ] \begin{pmatrix} I\otimes I \\ I\otimes I\end{pmatrix}\\
%     = & X_0 \otimes V_0^* V_0 + X_1\otimes (V_0^* V_1 + V_1^* V_0) + X_2\otimes V_1^* V_1.
% \end{split}
%\]
\end{proof}

\begin{lemma}
 \label{lem:invunderpi}
   Suppose $\tJ$ acts on the Hilbert space $\tE$ and is of the form in Equation
   \eqref{eq:Jordanc}. If $E$ is also a Hilbert space and
  $\pi:B(\tE) \to B(E)$ is a unital $*$-representation, then
   $J=\pi(\tJ)$ has, up to unitary equivalence,
  the form in Equation \eqref{eq:Jordanc} too.
\end{lemma}

\begin{proof}
 The operator $\tJ$ can be written as $\tW+c\tN$ where $\tW$ is unitary,
 $\tN^2=0$, $\tW\tN=\tN\tW$ and also $\tN^*\tN+\tN\tN^*=I$.   It follows that the same
  is true for $J =\pi(\tJ)$; i.e., $J= W +c N$
  where $W$ is unitary, $N$ is nilpotent of order two,
 $W$ and $N$ commute and $N^* N+ N N^*=I.$ 
  It is readily verified from these identities that $N N^*$
  and $N^* N$ are pairwise orthogonal projections. For instance,
\[
 N N^* = N (N N^* + N^* N) N^* = (N N^*)^2.
\]
  With respect to the orthogonal decomposition of  $E$
  determined by the projections $N N^*$ and $N^* N$
  and up to unitary equivalence,
\[
 N= \begin{pmatrix} 0 & I \\ 0 & 0 \end{pmatrix}.
\]
 Since $W$ commutes with $N$ it must have the form
\[
 W=\begin{pmatrix} U & 0 \\ 0 & U \end{pmatrix}.
\]
 Since $W$ is unitary, $U$ is unitary.  It now follows, that
  up to unitary equivalence, $J$ has the desired form. 
\end{proof}

\begin{proposition}
 \label{prop:premaininv}
    If $T$ in the class $\cFc$ and $T$ acts on the Hilbert space $H$,
  then there exists a Hilbert space $K$  an operator $J$ acting on $K$
  with the form in Equation \eqref{eq:Jordanc}
  and an isometry $V:H\to K$ such that $VT= J V$. 
\end{proposition}

\begin{proof}
 Choose $\cJ$ as in Equation \eqref{eq:scriptJ}.  By  Lemma \ref{lem:factor}, the
  mapping sending $\cJ^{*\alpha} \cJ^\alpha$ to $T^{*\alpha}T^\alpha$ is
  completely positive. Since $\cJ$ is symmetric and $T$ is invertible,
  Proposition \ref{prop:sym} implies 
  there exists a Hilbert space $K$ an isometry $V:\cK \to K$ and 
  a representation $\pi:B(\cK)\to B(K)$ such that $T^j V= V\pi(\cJ)^j$
  for all $j\in\mathbb Z$.  Finally, $J=\pi(\cJ)$ has the form
  of Equation \eqref{eq:Jordanc} by Lemma \ref{lem:invunderpi}.
\end{proof}

%\begin{proof}
%  Let $U$ denote the bilateral shift on $L^2$.  Thus $U$ is unitary and
%  $\sigma(U)=\mathbb T$.  Let $K$ denote the Hilbert space $L^2\oplus L^2$
%  and $J$ the operator on $K$ with block decomposition 
%\[
%  J = \begin{pmatrix} U & cU \\ 0 & U\end{pmatrix}.
%\]
%  
%  To show that
% the map $\tilde{\rho}$ mapping $\cH_s(\cJ)$ to $\cH_s(T)$
% as defined in Proposition \ref{prop:sym} is completely positive,
% first note that $\cH_s(\cJ)$ is spanned by
%$\{B_j(\cJ): j=0,1,2\}$ and $\cH_s(T)$ is likewise
%  spanned by $\{B_j(T): j=0,1,2\}$ and 
%  $\rho$ is determined by $\rho(B_j(\cJ)) = B_j(T)$.
%\end{proof}

\section{Lifting to an Invertible $3$-isometry}
 \label{sec:notinv}

   Theorem \ref{thm:main}, save for the equality of spectra,
   follows immediately from the following Proposition
  together with Proposition \ref{prop:premaininv}.

  \begin{proposition}
  \label{lem:lifttoinv}
    If $T\in B(H)$ is in the class $\cFc$, then there
  is a Hilbert space $K$, an operator $Y\in B(K)$ 
  such that $Y$ is invertible and in the
  class $\cFc$ and an isometry $V:H\to K$ such that
   $VT^n = Y^n V$ for all natural numbers $n$.   
\end{proposition}

  The proof of the proposition occupies the remainder
  of this section.  Given $T$ in $\cFc$, let
\[
 Q_+(T,s) = I + sB_1(T^*,T) + s^2 B_2(T^*,T).
\]
 Thus, $Q_+(T,s) = Q(T,s) + c^{-2} B_2(T^*,T)\succeq 0$. 

\begin{lemma}
 \label{lem:omnibus}
   If $T$ is in the class $\cFc$, then 
\begin{enumerate}[(i)]
  \item $\|B_2(T^*,T)\|\le c^2$: 
  \item $\|B_1(T^*,T)\| \le 2c$;
  \item $\|T\|\le 1+c$; and 
  \item $2(1+c^2)Q_+(T,s)-Q_+(T,s\pm 1) \succeq 0$.
\end{enumerate}
\end{lemma}

\begin{proof}
  Since $Q(T,0)\succeq 0$, it follows that $0\preceq B_2(T^*,T) \preceq c^2 I$
  and item $(i)$ follows. 

  Since $Q_+(T,s)$ is positive semidefinite for all $s$, for each vector $x$,
\[
 |\langle B_1(T^*,T)x,x\rangle |^2 \le 4 \langle B_2(T^*,T) x,x\rangle \, \langle x,x\rangle. 
\]
  Since all the operators involved are selfadjoint it follows that
\[
 \|B_1(T^*,T)\|^2 \le 4 \|B_2(T^*,T)\|.
\]
  Hence, $\|B_1(T^*,T)\| \le 2c$ in view of $(i)$. 

%  On the other hand,
%\[
% T^* T = I + B_1(T^*,T) +B_2(T^*,T).
%\]
 To prove item $(iii)$, observe,
\[
 \|T\|^2 =\|T^*T\|  \le 1 + \|B_1(T^*,T) \| + \|B_2(T^*,T)\| \le 1 + 2c + c^2 = (1+c)^2.
\]

 Straightforward computation reveals,
\[
2(1+c^2)Q_+(T,s)-Q_+(T,s\pm 1) = Q_+(T,s\mp 1)
   + 2c^2 Q(T,s)  \succeq 0,
\]
 proving item $(iv)$. 
\end{proof}

\begin{lemma}
 \label{lem:preliftinv1}
  If $T$ is a $3$-isometry, then for all 
  natural numbers $j$ and integers $n$,
\begin{equation}
 \label{eq:liftinv0}
  T^{*j} Q_+(T,n) T^j = Q_+(T,n+j).
\end{equation} 
  In particular,
\begin{equation}
 \label{eq:liftinv2}
  T^*B_2(T^*,T) T =B_2(T^*,T)
\end{equation}
 and 
\begin{equation}
 \label{eq:liftinv1}
 T^* B_1(T^*,T) T = \frac{T^{*2}T^2 - I}{2}.
\end{equation}
\end{lemma}

\begin{proof}
  Equation \eqref{eq:liftinv0} is evident in the case that $n$ is
  also a natural number.  
  Equations \eqref{eq:liftinv2} and \eqref{eq:liftinv1}
  follow from Equation \eqref{eq:3iso} and
  Equations \eqref{eq:B2} and \eqref{eq:B1} respectively.
  From here Equation \eqref{eq:liftinv0} follows
  from the definition of $Q_+(T,n)$.
\end{proof}

\begin{lemma}
 \label{lem:quadratic}
   Suppose $T\in B(H)$. If for each $h\in H$ there exists scalars
   $b_j(h)$  such that for all natural
   numbers $\alpha$, 
\[
   \langle T^{*\alpha}T^\alpha h,h\rangle = b_0(h) + \alpha b_1(h) + \alpha^2 b_2(h),
\]
  then $T$ is a $3$-isometry.  If moreover,
\[
    b_0(h) -\frac{b_2(h)}{c^2} + s b_1(h) + s^2 b_2(h) \ge 0
\]
 for each $h$ and all real $s$, then $T\in\mathcal F_c$. 
\end{lemma}

\begin{proof}
  The first hypothesis imply that for each fixed $h\in H$, 
\[
  \langle T^{*3}T^3-3T^{*2}T^2 +3T^*T -I)h,h\rangle =0.
\]
  By polarization, it now follows that $T$ is a $3$-isometry. 
  The second hypothesis is easily seen to imply $T$ is in the class $\cFc$.
\end{proof}

\begin{proof}[Proof of Proposition \ref{lem:lifttoinv}]
  Let $\cV$ \index{$\cV$} denote a  vector space (over $\mathbb C$)
  with countable basis $\{e_j: j\in\bZ\}$  and let
  Let $\cK$ denote the vector space $\cV \otimes H$
  and let   Define a sesquilinear form
  on $\cK$ by 
\[
  [e_m\otimes h, e_n \otimes k]
   = \begin{cases} \langle Q_+(T,n)T^{m-n} h,k\rangle & \mbox{ if } n\le m; \\
                          \langle T^{*(n-m)} Q_+(T,m)h,k\rangle & \mbox{ if } m \le n.
     \end{cases}
\]
  To see that this form is positive semi-definite, fix  positive
  integers $N$ and $M$  and let 
\begin{equation}
 \label{eq:h}
  h = \sum_{m=-N}^M c_m e_m\otimes h_m \in \cK
\end{equation}
 by given.  Note that by Lemma \ref{lem:preliftinv1}, if $-N\le n$, then
\[
  Q_+(n)=Q_+(-N+(n+N)) = T^{*(n+N)} Q_+(-N) T^{n+N}.
\]
 Thus
\begin{equation}
 \label{eq:hh}
 \begin{split}
 [h,h] = & \sum_{m,n=-N}^M c_m c_n^* [e_m\otimes h_m,e_n\otimes h_n]\\
   = & \sum_{-N\le m\le n\le M} c_m c_n^* \langle T^{*(n-m)}Q_+(T,m) h_m,h_n\rangle 
      + \sum_{-N\le n<m\le M} c_m c_n^* \langle Q_+(T,n) T^{m-n} h_m,h_n \rangle \\
  =&   \sum_{-N\le m,n\le M} c_m c_n^* \langle T^{*(n+N)}Q_+(T,-N)T^{m+N} h_m,h_n\rangle \\
  = & \langle Q_+(T,-N) g ,g\rangle \ge 0,
 \end{split}
\end{equation}
 where 
\[
g=\sum_n c_n T^{n+N} h_n =\sum_{j=0}^{N+M}  c_{j-N} T^j h_{j-N}
\]
 and  the very last inequality follows from the assumption that $Q(T,-N)\succeq 0$.
 Now let $K$ denote the Hilbert space obtained from $\cK$ by moding out null vectors
  and then forming the completion. 

 Define $Y:\cK\to\cK$ by
\[
  Y h  = \sum_{m=-N}^M c_m e_{m+1}\otimes h_m = \sum_{m=-(N-1)}^{M+1} c_{m-1} e_{m}\otimes h_{m-1},
\]
 where $h\in \cK$ as in Equation \eqref{eq:h}.  From equation \eqref{eq:hh}
\[
 [Yh,Yh] = \langle Q_+(T,-N+1) g,g \rangle.
\]
 Hence another applications of Equation \eqref{eq:hh}, the definitions and Lemma \ref{lem:omnibus} give
\[
  2(1+c^2)  [ Y h, Y h ] - [ h,h ] 
   = \langle \left ( (2+c^2) Q_+(T,-N+1)-Q_+(T,-N) \right ) g,g\rangle \ge 0.
\]
  Thus $Y$ determines a bounded operator on $K$ (denoted also by $Y$).  Similarly one finds 
\[
   2(1+c^2)  [  h,  h ] - [ Yh,Yh ]
   = \langle \left ( 2(1+c^2) Q_+(T,-N)-Q_+(T,-N+1)\right ) g,g\rangle \ge 0.
\]
 Hence $Y$ has a bounded inverse. 

 To see that $Y\in\mathcal F_c$, observe, for
 natural numbers $\alpha,$ and with $\dot{h}$ denoting the class of $h$ in $K$, 
\[
 \begin{split}
  \langle Y^{*\alpha} Y^\alpha  \dot{h},\dot{h} \rangle 
   = & [Y^\alpha h,Y^\alpha h] \\
    = & \langle Q_+(T,-N+\alpha) g,g\rangle \\
  %  = & \langle Q_+(T,-N)h,h\rangle + \alpha \langle (T^*T-2N T^{*2}T^2)h,h\rangle + \alpha^2 \langle T^{*2}T^2h,h\rangle
    = & \langle \left ( Q_+(T,-N) +\alpha(B_1(T^*,T)-2NB_2(T^*,T)) +\alpha^2 B_2(T^*,T)\right ) g,g\rangle
\end{split} 
\]
  and moreover,
\[
 \begin{split}
  \langle Q_+(T,-N)g,g\rangle 
     +& s\langle (B_1(T^*,T)-2NB_2(T^*,T))g,g\rangle  + s^2 \langle B_2(T^*,T)) g,g \rangle
          - \frac{\langle B_2(T^*,T)g,g\rangle}{c^2} \\
  = & \langle Q(T,-N+s) g,g\rangle \ge 0
 \end{split}
\]
  and apply Lemma \ref{lem:quadratic} to conclude 
   $Y\in\mathcal F_c.$ 

   Now suppose $n\in\mathbb Z$ and $h\in H$ and observe
\begin{equation}
 \label{eq:notmuch}
 \| (e_n\otimes Th) - (e_{n+1}\otimes h)\|^2
   =  \langle  T^*Q(n)Th,h\rangle
      - 2 \langle T^* Q(n)Th,h\rangle + \langle Q(n+1)h,h\rangle 
   =  0.
\end{equation}
  Thus, $e_n\otimes Th= e_{n+1}\otimes h$ in $\mathcal K$ (they represent the same equivalence class).
 To finish the proof, define $V:H\to K$ by
  $Vh=e_0\otimes h$.  From Equation \eqref{eq:notmuch}
\[
   VTh = e_0\otimes Th = e_1\otimes h = YVh
\]
 and thus $VT=YV$.
\end{proof}

\begin{corollary}
 \label{cor:boundnormT}
   If $T$ is in the class $\cFc$, then 
\[
  \|T\|^2 \le 1+\frac{c^2}{2} + c\sqrt{1+\frac{c^2}{4}}.
\]
\end{corollary}

\begin{proof}
 The norm of $J$ as in Equation \eqref{eq:Jordanc} is easily 
  seen to satisfy the inequality (with equality). The result 
  then follows from Theorem \ref{thm:main}.
\end{proof}

\section{Spectral Considerations}
 \label{sec:spectral}
   In this section it is shown that,
  in the setting of  Proposition \ref{prop:premaininv},
  the operator $J$ 
  can be chosen to satisfy
  $\sigma(J)= \sigma(T)$.   

\begin{proposition}
 \label{prop:spectral}
   Suppose $T\in B(H)$ is in the class $\cFc$.  If $T$ is invertible,
  then there is a Hilbert space $\cE$, unitary operator $W\in B(\cE)$
  and an isometry $V:H\to \cE \oplus \cE$ such that $\sigma(W)=\sigma(T)$ 
  and 
\begin{equation}
 \label{eq:spectral}
  VT=\begin{pmatrix} W & cW \\ 0 & W\end{pmatrix} V.
\end{equation}

  If $T\in B(H)$ is not invertible, then $\sigma(T)=\overline{\mathbb D}$. 
\end{proposition}

 Before turning to the proof of this proposition, we state the other
 main result of the section. 

\begin{proposition}[\cite{aglerstankusi}]
 \label{lem:specTnotinv}
   If $T$ is a non-invertible $3$-isometry, then $\sigma(T)=\overline{\mathbb D}$.
\end{proposition}

\begin{proof}
 Recall from the introduction that for any three isometry  $\sigma(T) \subset \overline{\mathbb D}$
  and the $3$-isometry $T$ is invertible if and only if $\sigma(T)\subset \mathbb T$.

 Suppose $\lambda \in\mathbb D$ and $T-\lambda$ is invertible. Let
\[
  S = (I-\ol T) (T-\lambda)^{-1}.
\]
 That $S$ is a $3$-isometry follows from directly calculating
\[
 \begin{split}
 0 = & (I-\lambda T^*)^3 (I-\ol T)^3 - 3 (T^*-\ol) (I-\lambda T^*)^2 (I-\ol T)^2 (T-\lambda)\\
      &   + 3 (T^*-\ol)^2 (I-\lambda T^*)(I-\ol T)(T-\lambda)^2
       - (T^*-\ol)^3 (T-\lambda)^3.
\end{split}
\]
  By inspection, $S$ is invertible. Thus $\sigma(S)\subset \mathbb T$.
   By spectral mapping the $\phi(\sigma(T)) = \sigma(S)$ where
\[
 \phi(\zeta) = (\zeta -\lambda) (1-\ol \zeta)^{-1}.
\]
 Hence $\sigma(T) \subset \mathbb T$ too.
\end{proof}

\begin{remark}\rm
 \label{lem:morespecT}
    In the case that $T\in\mathcal F_c$ (for some $c$)
    it is possible to use Theorem \ref{thm:main} to prove
    Proposition \ref{lem:specTnotinv}.  Indeed, 
    if $\phi$ is analytic in a neighborhood of $\obD$ and is
    unimodular on the boundary of $\mathbb D,$ then that
    $\phi(T)$ is a $3$-isometry can be seen from 
    $V\phi(T) = \phi(J)V$ and 
 \[
  \phi(J) = \begin{pmatrix} \phi(U) & U\phi^\prime(U) \\ 0 & \phi(U) \end{pmatrix},
 \]
  since in this case $\phi(J)$ is evidently a unitary plus commuting nilpotent of order two.
\end{remark}

 The remainder of this section contains a proof Proposition \ref{prop:spectral}
 and a brief subsection on the functional calculus for use in the next section

\subsection{A proof of Proposition \ref{prop:spectral}}
 Assuming $T$ is invertible, by Theorem \ref{prop:premaininv}, 
  there is a unitary operator $U$ acting on a Hilbert space $\cF$
 and an isometry $V:H\to\cF\oplus\cF$ such that $VT=JV$, where
\[
  J= \begin{pmatrix} U & cU \\ 0 & U \end{pmatrix}.
\]  
 The aim is to show that $U$ can be replaced by $W= (I-P)U(I-P)$, where
 $P$ is the spectral projection for the complement set $\sigma(T)$ associated to
 the unitary (normal) operator $U$ (so that $I-P$ is the spectral projection
 corresponding to $\sigma(T)$).

%  Recall that $J$ has the form
%  in Equation \eqref{eq:Jordanc}. If $\mathcal F$ is
%  the Hilbert space that
%  the unitary operator  $U$ acts upon, then $K=\mathcal F\oplus\mathcal F$
%  is the space that $J$ acts upon.

 Of course if $\sigma(T) =\mathbb T$, then
 there is nothing to prove.
 Otherwise, consider a nonempty closed arc $A$ in $\mathbb{T}\setminus \sigma(T)$ 
 and, for the purposes of this construction, 
 suppose the end points $\zeta^\prime$ and $\zeta^{\prime \prime}$ 
 of the arc $A$ are  equidistant from $\sigma(T).$
 We shall call this a centered arc. 
%More concretely, we let $d(\zeta^\prime,\sigma(T))=d(\zeta^{\prime \prime},\sigma(T))$.
 Let $\lambda$ be the midpoint of the arc, and choose a $t>1$. Consider the following diagram.
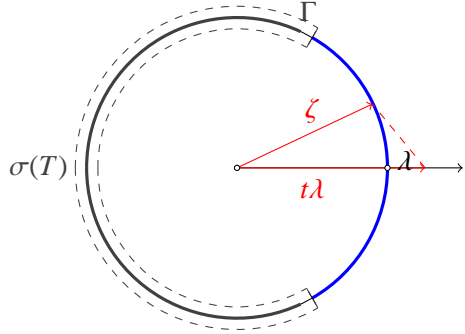
\begin{center}
\begin{tikzpicture}[scale=1]

%%unitcircle%%%%%%%%%%%%%%%%%%%%%%%%%%%%%%%%
\draw[black,thin](0,0) circle (2);
%%%%%%%%%%%%%%%%%%%%%%%%%%%%%%%%%%%%%%%%%%%%

%%ray%%%%%%%%%%%%%%%%%%%%%%%%%%%%%%%%%%%%%%%
\draw[black, ->](0,0)--(3,0);
%%%%%%%%%%%%%%%%%%%%%%%%%%%%%%%%%%%%%%%%%%%%

%%% arc%%%%%%%%%%%%%%%%%%%%%%%%%%%%%%%%%%%%%
\draw[blue, very thick](2,0) arc(0:60:2);
\draw[blue,very thick](2,0) arc(0:-60:2);
\draw[black](60:1.9)--(60:2.1);
\draw[black](-60:1.9)--(-60:2.1);
%%%%%%%%%%%%%%%%%%%%%%%%%%%%%%%%%%%%%%%%%%%%

%%% spectrum%%%%%%%%%%%%%%%%%%%%%%%%%%%%%%%%
\draw[darkgray,very thick](-2,0) arc(0:-115:-2);
\draw[darkgray, very thick](-2,0) arc(0:115:-2);
%%%%%%%%%%%%%%%%%%%%%%%%%%%%%%%%%%%%%%%%%%%%%

%% Contour%%%%%%%%%%%%%%%%%%%%%%%%%%%%%%%%%%%
\draw[darkgray,dashed](-1.85,0) arc(0:-120:-1.85);
\draw[darkgray,dashed](-2.15,0) arc(0:-120:-2.15);
\draw[darkgray, dashed](-2.15,0) arc(0:120:-2.15);
\draw[darkgray, dashed](-1.85,0) arc(0:120:-1.85);
\draw[darkgray, dashed](120:-1.85)--(120:-2.15);
\draw[darkgray, dashed](-120:-1.85) to (-120:-2.15);
%\draw[darkgray, dashed](-120:-2.15) arc [radius=.50,start angle=0, end angle=-45];
%%%%%%%%%%%%%%%%%%%%%%%%%%%%%%%%%%%%%%%%%%%%%%

%%%vectors%%%%%%%%%%%%%%%%%%%%%%%%%%%%%%%%%%%%
\draw[red, ->] (0,0)--(1.8, .85);
\draw[red, dashed](2.5,0)--(1.8, .85);
\draw[red, ->](0,0)--(2.5,0);
\node[below, red] at (1,0){\small $t\lambda$};
\node[right, black] at (1.98,.1){\small $\lambda$};
\node[above, red] at (1,.45){\small $\zeta$};
%%%%%%%%%%%%%%%%%%%%%%%%%%%%%%%%%%%%%%%%%%%%%%%

%% points%%%%%%%%%%%%%%%%%%%%%%%%%%%%%%%%%%%%%%
\draw[fill=white] (2,0) circle (0.035);
\draw[fill=white] (0,0) circle (0.035);
%%%%%%%%%%%%%%%%%%%%%%%%%%%%%%%%%%%%%%%%%%%%%%%

%%labels%%%%%%%%%%%%%%%%%%%%%%%%%%%%%%%%%%%%%%%
\node[left, darkgray] at (-2.1,0){\small $\sigma(T)$};
\node[left, darkgray] at (60:2.4){\small $\Gamma$};
%%%%%%%%%%%%%%%%%%%%%%%%%%%%%%%%%%%%%%%%%%%%%%%

%label arrows%%%%%%%%%%%%%%%%%%%%%%%%%%%%%%%%%%
%\draw[black, rounded corners, ->] (-2.5,-0.1)--(-2.5,-0.3)--(-2.25,-0.3)--(-2.05,0);
%%%%%%%%%%%%%%%%%%%%%%%%%%%%%%%%%%%%%%%%%%%%%%%
\end{tikzpicture}
\end{center}
A geometric argument shows, for $t>1$ fixed, that % that $|\zeta-t\lambda|$ for $\zeta\in A$ is largest at the endpoints of the arc, and thus 
\[
  \inf_{\zeta\in A}\left| \frac{1}{\zeta-t\lambda}\right|
     =\left|\frac{1}{\zeta^\prime-t\lambda}\right|
     = \left|\frac{1}{\zeta^{\prime\prime}-t\lambda}\right|.
\]
%where $\zeta^\prime$ is an endpoint of the arc. We then 
 Choose an $\alpha$ such that 
\[\left|\ \frac{\alpha}{\zeta_i-t\lambda}\right|=1\]
 and let $f(\zeta)=\alpha(\zeta-t\lambda)^{-1}$.
Choose a contour $\Gamma$ with $\sigma(T)$ on its inside and the arc $A$
 on its outside (the bounded and unbounded components determined by $\Gamma$ respectively)
  and such that the modulus of $f$ is less than one
 on and inside $\Gamma$.

   Since $f$ is analytic in a neighborhood of the
  closed unit disc and the spectrum of $J$ is 
  in $\mathbb T$, the expression $f(J)$ can be
  defined as a convergent power series.  On the
  other hand $f(U)$ can be defined in terms of
  power series or by the Borel functional calculus, 
  since $U$ is unitary (and hence normal).
  Of course both give the same value for $f(U)$.  It is
  straightforward to verify 
\[
  f(J) = \begin{pmatrix} f(U) & f^\prime(U) \\ 0 & f(U) \end{pmatrix}.
\]

  Now $f(T)$ can be defined as a convergent power series or 
 by the Riesz functional calculus,
\[
 f(T) = \frac{1}{2\pi i} \int_\Gamma f(z)(z-T)^{-1} dz.
\]

 Write, with respect to the decomposition $K=\mathcal F\oplus \mathcal F$,
\[
  V= \begin{pmatrix} V_1 \\ V_0 \end{pmatrix}.
\]
  Let $E$ denote the spectral measure for the unitary operator $U$.
 Thus, for any Borel set $B\subset \mathbb T$ the projection $E(B)$
  and $U$ commute and moreover,
\[
 U = \int_{\mathbb T} \lambda \, dE(\lambda).
\]

\begin{lemma}
\label{lem:Ben0}
  If $A$ is a closed centered arc such that the $A \cap \sigma(T)=\emptyset,$ then $E(A)V_\ell=0$ for $\ell=0,1$.
\end{lemma}

\begin{proof}
 From the Riesz functional calculus, $f^n(T)$ converges to $0$
 in the operator norm since $f^n$ converges to $0$ uniformly
  on $\Gamma$. 
 On the other hand,
\[
  Vf^n(T) = f^n(J)V
\]
  and hence $f^n(J)V$ also tends to $0$.

 Let $P$ denote
 the spectral projection (for $U$) corresponding to the arc $A,$
% Thus, %if $E$ is the spectral measure for $U$, then
\[
 P = \int_{A} dE = E(A).
\]
 Consider, with respect to the decomposition $K=\mathcal F\oplus \mathcal F$,
\[
 0\oplus P = \begin{pmatrix} 0 & 0 \\ 0 & P \end{pmatrix}
\]
  and similarly $P\oplus 0$.  Because $f^n(J)V$  tends to $0$ in operator norm, so do both
 \[
   V^* f^n(J)^* (0\oplus P) B_2(J^*,J) (0\oplus P) f^n(J) V
\]
 and
\[
  V^* f^n(J)^* (P\oplus 0)[I-\frac{1}{c^2}  B_2(J^*,J)] (P\oplus 0) f^n(J)V.
\]

 Straightforward computation shows
\[ 
 \begin{split}
  \frac{1}{c^2} f^n(J)^* (0\oplus P) B_2(J^*,J) (0\oplus P) f^n(J)
   = & \begin{pmatrix} f^n(U)^* & 0 \\ * & f^n(U)^* \end{pmatrix} 
       \begin{pmatrix} 0 & 0 \\ 0 & P \end{pmatrix} 
       \begin{pmatrix} f^n(U) & * \\ 0 & f^n(U) \end{pmatrix} \\
 = & \begin{pmatrix} 0 & 0 \\ 0 & f^n(U)^* P f^n(U) \end{pmatrix}.
\end{split}
\]
 It follows that $Pf^n(U)V_0$ tends to $0$. 
 On the other hand, $Pf^n(U)= f^n(U)P$ since $P$ is a spectral
  projection. 
 Consequently, using the Riesz functional calculus,
\[
 V_0^*P |f^n|^2(U) PV_0 =  V_0^* f^n(U)^* P f^n(U) V_0
\]
  also tends to $0$.  On the other hand, $P|f^n|^2 P \ge P$ since $|f^n|\ge 1$ on 
  the support $A$ of $P$.  Thus $PV_0 =0$; i.e., the range of $V_0$ lies in the
  range of $I-P$.

 Similarly,
\[
 f^n(J)^* (P\oplus 0) [I-\frac{1}{c^2} B_2(J^*,J)] (P\oplus 0) f^n(J)
   = \begin{pmatrix} P & 0\\0 & 0 \end{pmatrix}
    \begin{pmatrix} f^n(U)^* f^n(U) & * \\ * & * \end{pmatrix}
   \begin{pmatrix} P & 0\\0 & 0 \end{pmatrix}.
\]
 Hence, using the already established $PV_0=0$, 
\[
  V^* f^n(J)^* (P\oplus 0)(I-\frac{1}{c^2} (P\oplus 0) B_2(J^*,J)) f^n(J) V
  =   V_1^* P  f^n(U)^* f^n(U)  P V_1
\]
 converges to $0.$ Thus $PV_1=0$.
 \end{proof} 
 
 Note that
\[
  J V = \begin{pmatrix} W & cW \\ 0 & W \end{pmatrix} V,
\]
 where $W=(I-P)U(I-P)$ is a unitary operator on the Hilbert space
  $(I-P)\mathcal F$. 
   
\begin{lemma}
 \label{lem:Ben1}
   If $A\subset\mathbb T$ is a closed arc in the complement of $\sigma(T)$, then
  $E(A)V_\ell =0$ for $\ell=1,2$.
\end{lemma}

\begin{proof}
  Any such arc $A$ is contained in a closed centered arc $I$ disjoint from 
  $\sigma(T)$. Since $E(A)\preceq E(I)$ and, by Lemma \ref{lem:Ben0},
  $E(I)V_\ell=0,$ the conclusion of the lemma follows.
\end{proof}

\begin{lemma} 
\label{lem:Ben2}
  Suppose $A_1\subset A_2\subset$ is an increasing sequence of Borel subsets of 
 $\mathbb T$ and let  $A=\cup_j A_j$. 
  If $E(A_j)V_\ell=0$ for all $j$ and $\ell=0,1,$ then $E(A)V_\ell=0.$
\end{lemma}

\begin{proof}
 Consider the supremum $P$ of the projections $P_j=E(A_j).$
 Because $E$ is a spectral measure, $P=E(A)$.
 Since $P_j$ converges SOT (strong operator topology) to $P$,
 it follows that $P_j V_\ell$ converges to $PV_\ell$.
  Thus $PV_\ell =0$.
\end{proof}

 Now let $A$ be an open arc in the complement of $\sigma(T)$.  Since $A$ can be written as a disjoint union
 of closed arcs satisfying the hypothesis of Lemma \ref{lem:Ben1} it follows that 
  $E(A)V_\ell = 0$ for $\ell =1,2$ by Lemma \ref{lem:Ben2}. 

 Finally, $B=\mathbb T\setminus \sigma(T)$  is the disjoint union of 
 open arcs $\{A_j\}.$ %  whose closure is disjoint from $\sigma(T)$. 
 In particular, $E(A_j)V_\ell =0$ for each $j$ and $\ell$. Thus, another
 application of Lemma \ref{lem:Ben2} gives 
\[
 E(B) V_\ell = 0.
\]

 Let $P=E(B)$. The operator  $\fU=(I-P)U(I-P)$ is unitary and
\[
 \mathfrak{J} = \begin{pmatrix} \fU & c\fU \\ 0 & \fU \end{pmatrix}
\]
 has the form in Equation \eqref{eq:Jordanc}. Moreover,
 $\sigma(\fU) \subset \sigma(T)$ and hence $\sigma(\mathfrak{J})\subset \sigma(T)$ also.
 Of course,
\[
VT=  JV =  \mathfrak{J} V
\]
 too.

 It remains to show that $\sigma(\mathfrak{J})\supset \sigma(T)$.  To this end,
 suppose that $\lambda\in\mathbb T$ and $\mathfrak{J}-\lambda$ is invertible. 
  Let $L= V^*(\mathfrak{J}-\lambda)^{-1}V$ and observe that
  $L(T-\lambda)=I$ so that $T-\lambda$ is left invertible. On the other hand,
 choosing a sequence $\lambda_n$ not on the circle $\mathbb T$ but converging
  to $\lambda$ it follows that $(\mathfrak{J}-\lambda_n)^{-1}$ is bounded
  and converges to $(\mathfrak{J}-\lambda)^{-1}$.  Since
 $(T-\lambda_n)^{-1} =V^*(\mathfrak{J}-\lambda_n)^{-1})$, it follows 
  that $(T-\lambda_n)^{-1}$ converges to the left inverse $L$. 
  Thus $T-\lambda$ is invertible (and its inverse is $L$).

\subsection{More on the holomorphic functional calculus}
 \label{sec:funcalc}
  Let $\Omega$ denote an open simply connected subset of the plane. 
  Given an operator $T$ whose spectrum lies in $\Omega$ and a function
  $g$ holomorphic on $\Omega$, the operator $g(T)$ can be defined by
  the holomorphic (Riesz) functional calculus.  Moreover, if $T$ is normal,
  then so is $g(T)$. 
  Further, by Runge's Theorem, there exists a sequence of polynomials
  $p_n$ such that $p_n$ converges to $G$ uniformly on compact subsets of $\Omega$. 
  Thus by standard properties of the functional calculus, $p_n(T)$ converges to $g(T)$. 
  Likewise, $p_n^\prime$ converges uniformly to $g^\prime$ and that $p_n^\prime(T)$
  converges to $g^\prime(T)$. 

  In the special case that $\sigma(T)\subset \overline{\mathbb D} \subset \Omega$,
  the function $g$ has a power series expansion whose partial sums $(s_n)$  converges uniformly on compact subsets
  of $\Omega$. In particular, $s_n(T)$ converges to $g(T)$. 

   For $J$ as in Equation \eqref{eq:Jordanc} with $\sigma(U)$
  a subset of $\Omega$  and $p$ a polynomial, a simple calculation shows,
\[
  p(J) =\begin{pmatrix} p(U) & U p^\prime(U) \\ 0 & p(U) \end{pmatrix}.
\]
 In particular,
\begin{equation}
 \label{eq:fun-calc-J}
 g(J) = \lim p_n(J) = \begin{pmatrix} g(U) & U g^\prime(U) \\ 0 & g(U) \end{pmatrix}.
\end{equation}

   Likewise, if $A$ is selfadjoint with spectrum in $\Omega$, then with
\[
 \ttJ = \begin{pmatrix} A & cI \\ 0 & A \end{pmatrix},
\]
 we have
\[
  g(\ttJ) =\begin{pmatrix} g(A) & c g^\prime(A) \\ 0 & g(A) \end{pmatrix}.
\]
  Further, if $g(\sigma(A))\subset \mathbb T$, then $g(A)$ is normal with spectrum
  in the unit circle and is thus unitary. Consequently, $g(\ttJ)$ takes the
  form in Equation \eqref{eq:Jordanc}.

  Let $G$ denote the mapping $G(z)=\exp(iz)$, suppose
  $[a,b]$ of length strictly less than $2\pi$ and let 
  $S=G([a,b])$. In particular, $S$ is a proper subset of the unit circle $\mathbb T$.
  There exists open simply connected sets $\Omega$ and $\Omega_*$ containing
  $[a,b]$ and $S$ respectively such that $G:\Omega\to \Omega_*$ is bianalytic.
  If $\tT$ is an operator with spectrum in $[a,b]$, then $G(\tT)$ is well defined
  and has, by the spectral mapping theorem, its spectrum in $S$.  Letting
  $H$ denote the inverse of the mapping $G:\Omega\to\Omega_*$, the composition
  property of the holomorphic functional calculus implies that $H(G(\tT))=\tT$.

   Now suppose that $\tT$ is a $3$-symmetric operator with spectrum in $[a,b]$.
   In this case 
\[
  T=G(\tT)=\exp(i\tT)
\]
 is a $3$-isometry since $T^n=\exp(in\tT)$  for natural numbers $n$. Moreover,
 the spectrum of $T$ is a proper subset of the unit circle.

\section{$3$-Symmetric Operators}
 \label{sec:3symmetrics}
 Fix a $3$-symmetric operator $\tT\in B(H)$. 
%  From Equation \eqref{eq:3sym}, 
%\[
%  \tT^{*3}-3\tT^{*2}\tT +3 \tT^*\tT^2 -\tT^3 =0.
%\]
% (In fact this condition characterizes $3$-symmetric operators.)
 For a real numbers $s$ and  $t$,
\[
 \exp(ist\tT)^* \exp(ist\tT) = 
   I + s t \tB_1(\tT^*,\tT) + s t^2 B_2(\tT^*,\tT).
\]
 Thus $t\tT$ is also a $3$-symmetric operator and
\[
  B_j((t\tT)^*, t\tT)  = t^j \tB_j(\tT^*,\tT).
\]
 Let $c^2 = \|\tB_2(\tT^*,\tT)\|.$

  Choose a $t_0>0$ such that $\sigma(t_0\tT)$ is a subset
 of an interval $[a,b]$ of length less than $2\pi$. 
 Let 
\[
 S= \{\exp(is): s\in [a,b]\} \subsetneq \mathbb T.
\]
  Let $\Omega,$ $\Omega_*$ and $G$ be as at the end of
  Subsection \ref{sec:funcalc}.
 % Let $G$ denote the restriction of the mapping $z\mapsto \exp(iz)$
 % to  a  neighborhood $[a,b]\subset O \subset \mathbb C$
 % such that $G(O)$ is open and  $G$  has an analytic inverse (with
 % domain the open set $G(O)$). 
 In this case
\[
  T=\exp(it_0\tT) = G(T)
\]
 is a $3$-isometry with spectrum contained in $S.$     Moreover,
\[
  B_2(T^*,T) = t_0^2 \tB_2(\tT^*,\tT)
\]
 and thus,
\[
 (t_0 c)^2 = \|B_2(T^*,T)\|.
\]
 
   By Proposition \ref{prop:spectral},  there is a  Hilbert space $\cE$
  and a unitary operator $W\in B(\cE)$, an isometry 
  $V:H\to \cE\oplus \cE$ such that 
  $\sigma(T)=\sigma(W)$ such that Equation \eqref{eq:spectral}
  holds with $t_0c$ in place of $c$. Thus with
\[
   J= \begin{pmatrix} W & ct_o W \\ 0 & W\end{pmatrix},
\] 
  $VT=JV$.  Hence, as $G^{-1}$ is analytic in a neighborhood of the
 spectrum of $J$,
\[
   t_0 V\tT =  V  G^{-1}(T) = G^{-1}(J) V 
    = \begin{pmatrix} G^{-1}(W) & ct_0 W (G^{-1})^\prime(W) \\ 0 & G^{-1}(W) \end{pmatrix}.
\]
   Let $A=G^{-1}(W)$ and note that $(G^{-1})^\prime(W)= -i W^*$. Thus, with
\[
  \ttJ = \frac{1}{t_0} \begin{pmatrix} A & -i c t_0 \\ 0 & A \end{pmatrix},
\]
  $V\tT  = \ttJ V$
 and most of the following Theorem of Helton and Agler is established.

\begin{theorem}[\cite{aglerthreesym}\cite{heltonbulletin}\cite{heltonpaper}\cite{heltonpapererata}\cite{ballhelton}]
  If $\tT\in B(H)$ is a $3$-symmetric operator, but not selfadjoint, then
  $\tB_2(\tT^*,\tT)\neq 0$.  In this case, with $c=\|\tB_2(\tT^*,\tT)\|$, 
\begin{equation}
 \label{eq:existc}
 \exp(is\tT)^* \exp(is\tT) - \frac{\tB_2(\tT^*,\tT)}{c^2} \succeq 0
\end{equation}
 for all $s$. Moreover,  there exists a Hilbert space $\cE,$
 a selfadjoint operator $A\in B(\cE)$ and an isometry $V:H\to \cE\oplus \cE$ such that
\[
 VT = \begin{pmatrix} A & -i c\\ 0 & A\end{pmatrix}.
\]
\end{theorem}

 All that  remains to be proved is the inequality of  Equation \eqref{eq:existc}.
 To this end let $\tB_j = \tB_j(\tT^*,\tT)$, and note
\[
 \exp(is\tT)^* [I+t\tB_1 + t_2 \tB_2]\exp(is\tT) 
   = \exp(i(s+t)\tT)^* \exp(i(s+t)\tT) 
    = I + (s+t) \tB_1 +(s+t)^2 \tB_2,
\]
 from  which follows that
\[
  \exp(is\tT)^* \tB_2 \exp(is\tT)  = \tB_2.
\]
 Thus, with  $c^2=2\|\tB_2(\tT^*,\tT)\|$,
\[
 \exp(is\tT)^* \exp(is\tT) - \frac{\tB_2(\tT^*,\tT)}{c^2} =
   \exp(-is\tT^*)[I-\frac{\tB_2(\tT^*,\tT)}{c^2} ]
   \exp(is\tT) \succeq 0. 
\]

 Some of the results in this article were part of the
 first (alphabetically) listed authors research during
 his PhD studies at the University of California, San Diego,
 under the direction of Jim Agler.

\printindex

\end{document}